\documentclass[12pt]{l4dc2022}

\title[PGD Attacks for Region-of-Attraction Analysis of  Perception-Based Control]{Revisiting PGD Attacks for   Stability Analysis of  Large-Scale \\Nonlinear Systems and Perception-Based Control}
\usepackage{times}
\usepackage{wrapfig}
\newcommand{\norm}[1]{\left\|#1\right\|}
\newcommand{\field}[1]{\mathbb{#1}}
\newcommand{\R}{\field{R}}
\DeclareMathOperator*{\argmax}{arg\,max}
\DeclareMathOperator*{\maximize}{maximize}
\DeclareMathOperator*{\minimize}{minimize}
\DeclareMathOperator*{\sat}{sat}

\author{%
 \Name{Aaron Havens}\thanks{Equal Contribution} \Email{ahavens2@illinois.edu}\\
 \Name{Darioush Keivan}\footnotemark[1]  \Email{dk12@illinois.edu}\\
\Name{Peter Seiler} \Email{pseiler@umich.edu}\\
\Name{Geir Dullerud} \Email{dullerud@illinois.edu}\\
\Name{Bin Hu} \Email{binhu7@illinois.edu}
}

\begin{document}

\maketitle

\begin{abstract}%
  Many existing region-of-attraction (ROA) analysis tools find difficulty in addressing feedback systems with large-scale neural network (NN) policies and/or high-dimensional sensing modalities such as cameras.
 In this paper, we tailor the projected gradient descent (PGD) attack method developed in the adversarial learning community as a general-purpose ROA analysis tool for large-scale nonlinear systems and end-to-end perception-based control.
We show that the ROA analysis can be approximated as a constrained maximization problem whose goal is to find the worst-case initial condition which shifts the terminal state the most. Then we present two PGD-based iterative methods which can be used to solve the resultant constrained maximization problem. Our analysis is not based on Lyapunov theory, and hence requires minimum information of the problem structures. In the model-based setting, we show that the PGD updates can be efficiently performed using back-propagation. In the model-free setting (which is more relevant to ROA analysis of perception-based control), we propose a finite-difference PGD estimate which is general and only requires a black-box simulator for generating the trajectories of the closed-loop system given any initial state. We demonstrate the scalability and generality of our analysis tool on several numerical examples with large-scale NN policies and high-dimensional image observations.
We believe that  our proposed analysis serves as a meaningful initial step toward further understanding of closed-loop stability of large-scale nonlinear systems and perception-based control.
\end{abstract}

\begin{keywords}%
Region of attraction, PGD attack, perception-based control, neural networks
\end{keywords}
\section{Introduction}
\label{sec:intro}
In recent years, we have witnessed  tremendous  success of deep reinforcement learning (DRL) in various sequential decision-making applications \citep{silver2016mastering} and continuous  control tasks   \citep{lillicrap2015continuous,schulman2015high}. In the context of control, DRL has two main advantages.
First, DRL provides a general-purpose framework for addressing complex nonlinear dynamics (e.g. contact force, etc). For highly nonlinear dynamics,  one can still parameterize the controller as a neural network (NN) and apply DRL methods to train a useful controller.  Second,  DRL can be applied to train perception-based control systems in an end-to-end manner~\citep{lee2017learning,lee2019stochastic, hafner2019dream, yarats2019improving, yarats2021mastering}. One can parameterize the mapping from the pixel image to the control action as a neural network, and then it is straightforward to apply policy optimization to learn how to control from pixels. Despite these advantages, applications of DRL in real-world control systems are still rare. One issue is that the stability and robustness properties of such DRL-based controllers have not been fully understood. 
There is an urgent need to develop new analysis tools for addressing the stability and robustness of DRL-based control systems.

In this paper, we are interested in analyzing the asymptotic stability and estimating the region of attraction (ROA) of feedback control systems with large-scale NN policies and/or high-dimensional, rich observations (e.g. images).
There are three main technical difficulties. First, the NN policies can have a large number of hidden neurons. Second, the feedforward dynamics of the perception-based control systems are typically not fully known due to the complex mapping from the plant state to the image pixels. Third, in general, the control action  may not just be a function of the current state, and the coupling of system states at different time steps can be complicated. The first two issues will cause trouble for  existing Lyapunov-based ROA analysis methods using semidefinite programming \citep{yin2021stability,hu2020reach,jin2020stability,aydinoglu2021stability} or mixed-integer programs \citep{chen2020learning,chen2021learning,dai2021lyapunov}. Due to the last issue, the methods of Lyapunov neural networks \citep{richards2018lyapunov,chang2019neural,jin2020neural} or other stability certificate learning methods \citep{kenanian2019data,giesl2020approximation,ravanbakhsh2019learning} may also be not applicable since these methods typically require the control action to depend on the current state.
Our goal is to develop a ROA analysis method which can address the above three issues simultaneously.

To achieve our goal, we will borrow the method of projected gradient descent (PGD) attack developed in the adversarial learning  literature \citep{szegedy2013intriguing, goodfellow2014explaining, kurakin2016adversarial, madry2017towards} and tailor it as a general-purpose ROA analysis tool. Originally, 
the PGD attack was developed to find worst-case perturbation (overall some $\ell_p$ ball) that can shift the output of neural networks significantly and degrade the performances of classifiers in computer vision. It has been observed that such PGD-based methods work well for large-scale neural networks and vision-related tasks. In our paper, we build a connection between PGD attack and the ROA analysis. We show that the ROA analysis can be approximated as a constrained maximization problem whose goal is to find the worst-case initial condition (over some $\ell_p$ ball)  which shifts the terminal state most significantly.  Then PGD can be directly applied to solve the resultant maximization problem.  Such a maximization formulation is not based on Lyapunov theory, and hence does not require any particular structures for the underlying dynamical system.
Similar to the applications in computer vision, we find that PGD 
scales well and can address large-scale NN policies.  When the unknown mapping from the plant state to the image pixels and complex coupling between states at different time steps are involved, we propose a finite difference PGD estimation which is general in the sense that it only requires a black-box simulator for generating the trajectories of the closed-loop systems given any initial state. 
Consequently, the proposed method can address the three challenges mentioned above simultaneously and provide a meaningful initial step towards scalable stability analysis for feedback systems with large-scale NN policies and/or high-dimensional, rich observations.
Finally, we present some numerical experiments as well as some concluding remarks.

%


\section{Problem Formulation}
%

In this paper, we are interested in stability analysis of complex nonlinear control systems.
Specifically, consider the following nonlinear dynamical system 
\begin{align}
\label{eq:sys1}
\begin{split}
x_{t+1}&=f(x_t, u_t)\\
y_t& =h(x_t)
\end{split}
\end{align}
where $x_t\in \R^{n_x}$ is the state, $u_t\in\R^{n_u}$ is the control input, and $y_t\in \R^{n_y}$ is the output observation. We are interested in the general nonlinear control setting where $u_t$ is determined by a complex nonlinear mapping $K$ from the history of observation-action pairs over a time window, i.e. $u_t=K\left(y_t, y_{t-1}, u_{t-1}, \ldots,y_{t-N+1}, u_{t-N+1}\right)$ where $N$ is the window length.
We allow the analytical form of  the function $h$ to be unknown and hence $y_t$  is allowed to be a high-dimensional, rich observation from a camera.
 In this case, $y_t$ is just a vector augmented from the image pixel values obtained at time $t$, and we make the assumption that the environment for the image generation is relatively static such that $h$ is deterministic.  For perception-based control systems, the analytical form of $h$, is typically unknown. However, it is reasonable to assume that we have access to an image generator which can simulate the output of $h$ for a given state.

For simplicity, we assume that the system \eqref{eq:sys1} is posed in a way that the equilibrium state is $0$. Our goal is to
analyze the asymptotic stability and estimate the region of attraction (ROA) of the feedback interconnection of \eqref{eq:sys1} and $K$.
Clearly, for any fixed $t$, the state $x_t$ of the closed-loop feedback control system \eqref{eq:sys1} with policy $K$ will be uniquely determined by a mapping from the initial state $x_0$. We denote such a mapping as $g_t$. 
Once $f$, $h$, and $K$ are fixed, $g_t$ will become determined for any $t$.
Then the state trajectory generated by the closed-loop feedback control system satisfies $x_t=g_t(x_0)$ for any $t$.
Now we can define ROA as follows.

\begin{definition}\label{def:roa}
The ROA of the feedback control system \eqref{eq:sys1} with the policy $K$ is defined as
\begin{align}
\mathcal{R}=\{x_0: \lim_{t\rightarrow \infty}g_t(x_0)=0\}
\end{align}
\end{definition}

In this paper, we are interested in finding convex approximations of $\mathcal{R}$ and addressing the following two difficult cases.
\begin{enumerate}
\item Policies parameterized by large-scale neural networks (NNs):  In deep RL, it is popular to parameterize the policy $K$ as a large-scale neural network. The number of the neurons can be fairly large, and this cause a scalability issue for performing ROA analysis. 
\item Control from pixels: For perception-based control systems, it is difficult to figure out the internal mechanism of image generation, and hence the mapping $h$ is typically unknown. Notice that the output of $h$ is typically a high-dimensional signal, and fitting a function to estimate such $h$ is also difficult. In addition, one will need to augment images from different time steps to determine the current control action, and this will also lead to a coupling effect which causes trouble for analysis.
\end{enumerate}

In general, it is very difficult to obtain tight rigorous approximations of $\mathcal{R}$ for large-scale nonlinear/perception-based control systems. In this paper, we borrow the idea of PGD attack to generate initial conditions which do not belong in $\mathcal{R}$ and then construct reasonable ROA approximations using these initial conditions.

\section{Main Analysis Framework}

\subsection{A Constrained Maximization Formulation for Finite-Horizon Approximations of ROA}
In this paper, we are interested in approximating $\mathcal{R}$ as the following parameterized convex set
\begin{align}
\hat{\mathcal{R}}(p, r, C)=\{\xi: \norm{C\xi}_p\le r\}
\end{align}
where $C$ is some prescribed transformation matrix, $r$ quantifies the size of the approximated ROA, and $p$ can be $1$, $2$, or $\infty$.
Notice that $C$ should be full rank such that $C^\top C$ is a positive definite matrix.
When $p=2$, we will have ellipsoidal approximations. Clearly, the set $\hat{\mathcal{R}}(p,r,C)$ is always convex, and hence projection to $\hat{\mathcal{R}}(p,r,C)$ can be easily done. We also want to mention that for convenience, we will drop the subscript ``$2$" in the notation of  the $\ell_2$ norm and just use $\norm{\cdot}$ instead.

If $\hat{\mathcal{R}}(p, r, C)\subset \mathcal{R}$, then we have $\lim_{t\rightarrow \infty}g_t(\xi)=0$ for any $\xi\in \hat{\mathcal{R}}$. Therefore, a necessary and sufficient condition for $\hat{\mathcal{R}}(p, r, C)\subset \mathcal{R}$ is given as follows
\begin{align}
 \max_{\xi\in \hat{\mathcal{R}}(p, r, C)}\left(\limsup_{t\rightarrow \infty}\norm{g_t(\xi)}^2\right)=0.
\end{align}

Checking the above condition numerically will lead to a finite-horizon approximation:
\begin{align}
\label{eq:cond1}
 \max_{\xi\in \hat{\mathcal{R}}(p, r, C)}\norm{g_T(\xi)}^2\le \delta
\end{align}
where $T$ is a prescribed large number and $\delta$ is some fixed small number. We will use \eqref{eq:cond1} as our main criterion for approximating $\mathcal{R}$. Specifically, given $C$, $p$, $r$, $T$, and $\delta$, we will calculate  $ \max_{\xi\in \hat{\mathcal{R}}(p, r, C)}\norm{g_T(\xi)}^2$ and compare the resultant solution with $\delta$.
For a fixed $C$ and $p$, we will perform a bisection on $r$ to find the maximum of $r$ such that \eqref{eq:cond1} is satisfied. Then the resultant $\hat{\mathcal{R}}(p,r,C)$ will be our ROA approximation. 

We formalize the above discussion by defining the $(T,\delta)$-approximated region of attraction (AROA) as follows.
\begin{definition}
\label{def:AROA}
The $(T,\delta)$-AROA is defined as
\[
\tilde{\mathcal{R}}(T,\delta)=\{x_0: \norm{g_T(x_0)}^2\le \delta\}.
\]
\end{definition}


The following lemma gives  a precise characterization of
the relation between $\tilde{\mathcal{R}}(T,\delta)$ and $\mathcal{R}$.
\begin{lemma}
For any fixed $T$ and $\delta>0$, we have
\begin{align}
\label{eq:unionR}
\bigcap_{t\ge T}\tilde{\mathcal{R}}(t,\delta)=\{x_0:\norm{g_t(x_0)}^2\le \delta,\,\,\,\forall t\ge T\}.
\end{align}
The sequence of sets $\left\{\bigcap_{t\ge T}\tilde{\mathcal{R}}(t,\delta)\right\}_{T=0}^\infty$ is monotonically increasing to $\liminf_{T\rightarrow \infty}\tilde{\mathcal{R}}(T,\delta)$.
In addition, $\liminf_{T\rightarrow \infty}\tilde{\mathcal{R}}(T,\delta)$ is monotonically decreasing in $\delta$, and the following limit holds
\begin{align}
\label{eq:Rlimit}
\mathcal{R}=\lim_{\delta\rightarrow 0}\left(\liminf_{T\rightarrow \infty}\tilde{\mathcal{R}}(T,\delta)\right).
\end{align}
\end{lemma}
\begin{proof}
Most statements in the above lemma can be verified trivially. The proof of \eqref{eq:Rlimit} is less straightforward and hence included here. First, we will show $\mathcal{R}\subset\lim_{\delta\rightarrow 0}\left(\liminf_{T\rightarrow \infty}\tilde{\mathcal{R}}(T,\delta)\right)$. Suppose $x_0\in \mathcal{R}$. By definition, we have $\lim_{t\rightarrow \infty} g_t(x_0)=0$. Therefore, for any $\delta_0>0$, there exists $T$ such that $\norm{g_t(x_0)}^2\le \delta_0$ for all $t\ge T$. This means $\mathcal{R}\subset \liminf_{T\rightarrow \infty}\tilde{\mathcal{R}}(T,\delta_0)$ for any $\delta_0>0$. Then we can let $\delta_0$ approach $0$ and have $\mathcal{R}\subset\lim_{\delta\rightarrow 0}\left(\liminf_{T\rightarrow \infty}\tilde{\mathcal{R}}(T,\delta)\right)$. Next, we will show $\lim_{\delta\rightarrow 0}\left(\liminf_{T\rightarrow \infty}\tilde{\mathcal{R}}(T,\delta)\right)\subset \mathcal{R}$. Suppose $x_0\in \lim_{\delta\rightarrow 0}\left(\liminf_{T\rightarrow \infty}\tilde{\mathcal{R}}(T,\delta)\right)$. For any arbitrary $\delta_0>0$, it is straightforward to verify $ \lim_{\delta\rightarrow 0}\left(\liminf_{T\rightarrow \infty}\tilde{\mathcal{R}}(T,\delta)\right)\subset \liminf_{T\rightarrow \infty}\tilde{\mathcal{R}}(T,\delta_0)$. This means that for any arbitrary $\delta_0>0$, there exists $T>0$ such that $\norm{g_t(x_0)}^2\le \delta_0$ for $t\ge T$. Therefore, we have $\lim_{t\rightarrow \infty} g_t(x_0)=0$. This leads to the desired conclusion and completes the proof.
\end{proof}

Based on \eqref{eq:Rlimit}, it is reasonable to  estimate $\mathcal{R}$ from $\tilde{R}(T,\delta)$ with some small $\delta$ and large $T$. From a practical point of view,  Definition \ref{def:AROA} also makes sense since engineering systems are run on finite-time windows. Stability and ROAs defined on infinite horizons provide meaningful abstractions for quantifying the resilience property of many feedback control systems. However, we will show that the finite-horizon notion of ROA will provide complementary benefits from a computational perspective. Specifically, verifying whether $\hat{\mathcal{R}}(p,r,C)\subset \tilde{\mathcal{R}}(T,\delta)$ is equivalent to a constrained maximization problem which can be approximately solved in a scalable manner using gradient-based methods.
Our finite-horizon approach will not give a rigorous inner approximation of $\mathcal{R}$. However, if we choose $T$ and $\delta$ carefully, the above constrained maximization approach will lead to scalable solutions for estimating ROA of complex nonlinear systems and perception-based control. We will elaborate on this point later.

\subsection{PGD Attack for ROA Approximation}\label{sec:pgd_attack}
From the above discussion, the ROA analysis is formulated as the following maximization problem
\begin{align}
\label{eq:maxGT}
\maximize_{\xi\in \hat{\mathcal{R}}(p,r,C)} \norm{g_T(\xi)}^2
\end{align}
Denote $\xi^*=\argmax_{\xi\in \hat{\mathcal{R}}(p,r,C)} \norm{g_T(\xi)}^2$.
We will perform bisection on $r$ to find the largest $r$ such that $\norm{g_T(\xi^*)}^2\le \delta$, and then we  will use the resultant set $\hat{\mathcal{R}}(p, r, C)$ to approximate the ROA. 

The key to our analysis is that we can apply PGD to solve $\xi^*$. 
Denote $L_T(\xi)=\norm{g_T(\xi)}^2$.
In addition, for any convex set $S$, we use $\Pi_S$ to denote the projection onto $S$.
Then PGD iterates as\footnote{Technically speaking, projected gradient ascent is needed for maximization problems. However, the terminology PGD is still used here such that our paper is consistent with the adversarial learning literature.}
\begin{align}
\label{eq:PGD}
\xi^{k+1}=\Pi_{\hat{\mathcal{R}}(p,r,C)}\left(\xi^k+\alpha \nabla L_T(\xi^k)\right).
\end{align}
One way to interpret \eqref{eq:PGD} is that it recursively solves the following approximated form of \eqref{eq:maxGT}:
\begin{align}
\label{eq:expan1}
\minimize_{\xi\in \hat{\mathcal{R}}(p,r,C)} \left\{-L_T(\xi^k)-\nabla L_T(\xi^k)^\top (\xi-\xi^k)+\frac{1}{2\alpha}\norm{\xi-\xi^k}^2\right\}
\end{align}
where an $\ell_2$ regularizer is added to the first-order Taylor expansion of $(-L_T)$ around $\xi^k$. The exact solution for \eqref{eq:expan1} is given by $\Pi_{\hat{\mathcal{R}}(p,r,C)}\left(\xi^k+\alpha \nabla L_T(\xi^k)\right)$ which is used as the next iterate $\xi^{k+1}$.

Another popular way to solve \eqref{eq:maxGT} recursively is to approximate \eqref{eq:maxGT} as
\begin{align}
\label{eq:expan2}
\xi^{k+1}=\argmax_{\norm{C\xi}_p^2=r^2} \left\{L_T(\xi^k)+\nabla L_T(\xi^k)^\top (\xi-\xi^k)\right\}
\end{align}
where the $\ell_2$ regularizer is removed and the inequality constraint $\norm{C\xi}_p^2\le r^2$ is replaced with an equality condition  $\norm{C\xi}_p^2=r^2$. 
Using an equality constraint makes sense for the ROA analysis since we will perform a bisection on $r$ and the ROA of interests is typically in the form of continuum. 
When $p=2$, the constraint is just $\norm{C\xi}^2= r^2$, and a closed-form solution for \eqref{eq:expan2} is given as
\begin{align}
\label{eq:PI}
\xi^{k+1}=\frac{r}{\norm{ C^{-\top}\nabla L_T(\xi^k)}} (C^\top C)^{-1} \nabla L_T(\xi^k)
\end{align}
To see this, we apply the Lagrange multiplier theorem to \eqref{eq:expan2} and obtain the following conditions
\begin{align*}
-\nabla L_T(\xi^k)+2\lambda C^\top C\xi&=0\\
\norm{C\xi}&=r
\end{align*}
where $\lambda$ is the Lagrange multiplier. Since we are solving a maximization problem, we can immediately obtain the formula \eqref{eq:PI}. The above update is typically more efficient since it exploits the prior intuition that the worst-case initial points should be on the boundary for the maximized $r$. There is also a deep connection between \eqref{eq:PI} and the so-called alignment condition in the controls literature \citep{tierno1997numerically}, since the initial condition can be viewed as an input applied at the initial step.

For $p=2$, the implementation of \eqref{eq:PI} is straightforward. For other values of $p$, \eqref{eq:expan2} can also be applied. We discuss the update rules for theses cases in the appendix. Another subtle issue is how to choose $T$. Notice that $T$ cannot be too large. Otherwise the gradient $\nabla L_T(\xi)$ may be too small for any $\xi\in \mathcal{R}$ and this makes finding $\xi^*$ more difficult. Decreasing $T$ is actually smoothing the cost function $L_T$ and makes the optimization easier. However, we also cannot make $T$ be too small. Otherwise $\tilde{\mathcal{R}}(T,\delta)$ is no longer a good estimate for $\mathcal{R}$.  We will discuss more in our numerical study.

\begin{remark}\label{rem1}
Our proposed analysis only provides an approximation for the true ROA. Obviously, there is a gap between $\mathcal{R}$ and $\tilde{\mathcal{R}}(T,\delta)$. If $T$ and $\delta$ are well chosen, the approximation error induced by such a gap will be small. Then we can just verify whether $\hat{\mathcal{R}}(p,r,C)\subset\tilde{\mathcal{R}}(T,\delta)$, and
the optimization error in solving \eqref{eq:maxGT} will become a more dominant factor.  In general, \eqref{eq:maxGT} is a non-concave maximization problem. If $\hat{\mathcal{R}}(p,r,C)\not\subset\tilde{\mathcal{R}}$, we can verify this by finding one initial condition satisfying $L_T(\xi)>\delta$. This is relatively easy since this case does not require us to solve the non-concave maximization problem \eqref{eq:maxGT} exactly. It is expected that PGD with random initialization can provide an efficient search method to find such initial conditions. However, to verify $\hat{\mathcal{R}}(p,r,C)\subset\tilde{\mathcal{R}}$,  we need to find the global maximum of \eqref{eq:maxGT} and compare it with $\delta$. In general, there lacks strong theoretical guarantees for finding the global solutions for such non-concave maximization problems. 
One useful heuristic fix is to run PGD with different random initial conditions and treat the worst-case value as an approximate solution for \eqref{eq:maxGT}. Despite the lack of strong global guarantees, our simulation study shows that the ROA approximations from PGD  are good estimates of the true ROAs in many situations. More theoretical study is needed to better understand our numerical findings.
\end{remark}

Both \eqref{eq:PGD} and \eqref{eq:PI} are inspired by existing attack methods in the adversarial learning literature \citep{madry2017towards,goodfellow2014explaining}. The sign gradient information is typically needed for image classification problems while our ROA analysis will directly use the true gradient. 
For both \eqref{eq:PGD} and \eqref{eq:PI}, the key step is to evaluate $\nabla L_T(\xi^k)$.  Next, we discuss how to perform such gradient evaluation for nonlinear systems with large-scale NN policies and/or image observations.

\subsection{Model-Based ROA for Large-Scale NN Policy: PGD with Back-Propagation}

First, as a sanity check, we consider the relative simple case where $u_t=K\circ h(x_t)$ with the analytical forms of both $K$ and $h$ being known apriori. Here the operation $\circ$ denotes the composition of two maps.
In this case, we can use back-propagation to evaluate $\nabla L_T(\xi)$ efficiently. 
The feedback system reduces to the autonomous form $x_{t+1}=f(x_t, K\circ h(x_t))$. For simplicity, we denote $\tilde{f}(x)=f(x, K\circ h(x))$. 
We also denote $\tilde{f}^{(0)}$ to be the identity map and set $\tilde{f}^{(n+1)}=\tilde{f}\circ \tilde{f}^{(n)}$.
Then we have $x_t=\tilde{f}^{(t)}(x_0)$, and $L_T(x_0)=\norm{\tilde{f}^{(t)}(x_0)}^2$. Then $\nabla L_T(\xi)$ can be evaluated using back-propagation. Specifically, we can introduce the costate $p_t$, and
 define the Hamiltonian as
$H(x, p):=p \tilde{f}(x).$
Then the following result holds.
\begin{lemma}
Set $x_0=\xi$ and generate the state sequence as $x_{t+1}=\tilde{f}(x_t)$ for $t=0, 1,2, \ldots, T-1$. Next, set $p_T=2x_T$ and generate the costate sequence in a backward manner, i.e.
\[
p_t=\nabla_x H(x_t, p_{t+1})\,\,\,\,\,\mbox{for} \,\,\,\,0\le t\le T-1
\]
Then we have $\nabla L_T(\xi)=p_0$.
\end{lemma}
\begin{proof}
The above result can be viewed as a special case of Proposition 5 in \cite{li2018maximum}. One can also directly apply the chain rule to verify the above result. The details are omitted. 
\end{proof}
The above result states that one can just run the forward dynamics and then make a back-propagation to calculate the $\nabla L_T(\xi)$. When $f$, $h$, and $K$ are known, one can write out explicit expressions for $\nabla_x H$ and hence the above gradient evaluation can be efficiently performed.

\begin{remark}
If $f$ is a linear function and $K\circ h$ is a neural network, one can further evaluate $\nabla_x H$ using back-propagation. Tools such as PyTorch or JAX can be used to perform the auto-differentiation~(\cite{paszke2019pytorch, jax2018github}). Such an evaluation can be easily scaled up to address NN policies with more than thousand neurons. 
\end{remark}

Based on the above lemma, we can combine either \eqref{eq:PGD} or \eqref{eq:PI} with the back-propagation rule to solve $\xi^*$.
For perception-based control systems, the analytical form of the mapping $h$ is typically not known. 
In addition, $y_t$ may depend on the past state and action. The above back propagation approach may no longer work. 
One may argue that one can try to fit a model on $h$. However, such a fitting may not be always feasible. In addition, even if we can fit such a model, the dimension of $y_t$ can be high and the required back-propagation step becomes much more expensive. 
This motivates us to use a model-free approach introduced next.

\subsection{Model-Free ROA Analysis: Derivative-Free Optimization for Perception-Based Control}
In this section, we present a model-free approach to address the case where the analytical forms of $f$ and $h$ are unknown but a black-box simulator for $g_T$ is available. Specifically, we assume that we can simulate the closed-loop system and generate $g_T(\xi)$ for any given $T$ and $\xi$. For perception-based control systems, the dimension of $y_t$ is high. However, a key fact is that $g_T(\xi)$ is a function of $\xi$ which lives in a space with much lower dimension. This motivates us to apply the finite difference estimation for the gradient evaluation. Specifically, we have $\nabla L_T(\xi^k)\approx \Gamma$ where the $j$-th entry of $\Gamma$ can be estimated using the following finite difference scheme:
\begin{align}
\label{eq:finitedifference}
\Gamma(j)=\frac{L_T(\xi^k+\epsilon e_j)-L_T(\xi^k)}{\epsilon}.
\end{align}
Notice $e_j$ denotes a vector whose entries are
all $0$ except the $j$-th entry which is $1$. Therefore, we need to simulate the trajectories for $(n_x+1)$ times. For systems with reasonable state dimension, such a gradient evaluation is scalable. 

Notice that the above approach is not useful when trying to find the worst-case attack on the observation $y_t$. The dimension of $y_t$ is really high, and a finite difference estimation  will not be scalable on the space of $y_t$. For the purpose of ROA analysis, such an approach scales with the dimension of $\xi$ and can generalize as long as $n_x$ is not too big. 

The above finite difference method is extremely general. In principle, as long as we can have a black-box simulator for the closed-loop system, we can apply the finite difference method without knowing any underlying dynamic structures. Such an approach is mostly useful for perception-based control systems with complex mapping $h$.

\section{Experimental Results}
In this section, we present numerical results to demonstrate the effectiveness of our proposed PGD-based analysis in the following two settings. 
\begin{itemize}
    \item Perception-based Control: We consider several examples where control actions are directly determined from RGB-image pixels (see Figure \ref{fig:image_feedback}). In Section \ref{sec:inverted_pendulum}, we present our ROA analysis results for perception-based control of nonlinear inverted pendulum. In Section \ref{sec:cartpole}, perception-based control of cartpole systems (with single and double links) is studied.  The polices for these problems are parameterized using large-scale NNs, and our results show that our PGD approach is highly-flexible and effective for perception-based control. As a sanity check, for the inverted pendulum problem, we have also included some analysis results for the relatively simple state-feedback setting
where our approach can be justified via a comparison with the Lyapunov-based inner approximation approach developed in \citet{yin2021stability}.

    \item Nonlinear System Analysis with High-Dimensional States: We also study how our proposed approach scales with the system state dimension. In Section \ref{sec:cubic}, we consider a cubic systems example with a known ROA.
Unlike the sum-of-squares (SOS) approaches, our PGD approach can scale reasonably to higher state dimensions. We characterize the scaling capability of our approach up to the case where the state dimension is $1000$.
\end{itemize}

\begin{wrapfigure}{r}{0.5\textwidth}
    \centering
    \includegraphics[width=0.5\textwidth]{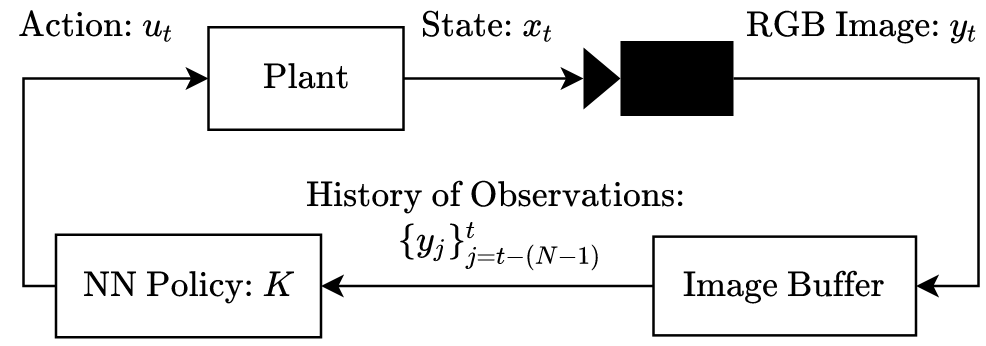}
    \caption{A diagram depicting the image-based feedback setup.}
    \label{fig:image_feedback}
\end{wrapfigure}
For both settings, we will apply the PGD update rule \eqref{eq:PI} with the finite difference gradient estimator \eqref{eq:finitedifference} to perform ROA analysis. More discussions on the alternative PGD update rule \eqref{eq:PGD} will be given in the appendix.
Before proceeding to the concrete examples below, we briefly clarify the perception-control setting considered in our paper. The perception-based feedback control loop is visualized in Figure~\ref{fig:image_feedback}.
The plant can be any nonlinear systems (e.g. inverted pendulum, cartpole, etc). A camera is used to measure the system output. The environment is assumed to be static such that the mapping from the system state to the RGB images is time invariant. This is a reasonable assumption, and similar settings have been adopted in many simulation environments for perception-based control  \citep{tunyasuvunakool2020dm_control,xu2021learned}.
The controller uses the last $N$ $84\times 84$ RGB images, which are generated by the \textit{Deepmind Control Suite}~\citep{tunyasuvunakool2020dm_control}. 
 We train each perception-based controller using the novel image-augmentation training procedure from~\cite{yarats2021mastering}, and utilize the model-free RL algorithm Soft-Actor-Critic (SAC)~\citep{haarnoja2018soft} as the policy optimizer. 
Within the controller $K$,
a policy network is prepended by a four-layer CNN encoder with $3\times 3$ kernels and $32$ channels, applying ReLU activations at each CNN layer. The output of the CNN encoder is normalized and then fed to the fully-connected four-layer ReLU policy network with $1024$ neurons.

\subsection{Inverted-Pendulum}\label{sec:inverted_pendulum}
Now we apply our proposed ROA analysis to the nonlinear inverted-pendulum problem.
 Although the plant dynamics are relatively simple, the output $y_t$ could be some high-dimensional and complex function of $x_t$. It is convenient for exposition to look at this $2D$ system since the true ROA can be readily evaluated and visualized. 
We discretize the well-known inverted pendulum dynamics using a simple Euler scheme with a sample time of $dt=0.02$s, and obtain the following state-space model
\begin{align}\label{eq:inverted_pen_dyn}
    \begin{bmatrix}
    \theta_{t+1}\\
    q_{t+1}
    \end{bmatrix}
    = 
    \begin{bmatrix}
    \theta_t + q_t dt\\
    q_t +\left( \frac{g}{l}\sin \theta_t - \frac{\mu}{m l^2} q_t + \frac{1}{m l^2} \sat(u_t)\right) dt
    \end{bmatrix},
\end{align}
where $\theta_t$ is the pendulum angle, $q_t$ is the angular velocity, $u_t$ is the control action, $m$ is the pendulum rod mass, $l$ is the length, $\mu$ is a damping coefficient, and $\sat(\cdot)$ is a saturation function capturing the saturation limit on the control action. We augment $(\theta_t, q_t)$ to get the state $x_t$. 
The controller is designed in a way such that  the equilibrium point for the closed-loop system is $x_e=0$. Our goal is to estimate the ROA by searching for the worst-case initial condition $x_0$ via our proposed PGD method. As a sanity check, we will first study the state-feedback setting where $u_t=K(x_t)$ with $K$ being some fully-connected NN, and compare our PGD analysis against a recently-developed quadratic constraint approach. Then we consider the image-based setting where an even larger NN policy is learned using the previous two RGB images directly as input, i.e. $u_t=K(y_t, y_{t-1})$ where $K$ consists of a policy network and a CNN encoder.
%

As pointed out in Remark \ref{rem1},
our approach only provides an approximation for the true ROA. Due to the gap between $\mathcal{R}$ and $\tilde{\mathcal{R}}(T, \delta)$, our approach can not provide rigorous inner approximation of the true ROA.
Although our approach
 does not
provide a provable certificate for the traditional notion of Lyapunov stability,
we will see that our PGD-based ROA analysis is often fairly tight and does lead to reasonable approximations of the true ROA.

\paragraph{Sanity check: Comparison with the quadratic constraint approach.} 
As a sanity check, we first provide some results for the state-feedback setting where our approach can be justified via a comparison with the quadratic constraint approach developed in ~\cite{yin2021stability}. The experimental setup is identical to Example IV.A in ~\cite{yin2021stability}, and we adopt the same policy for the purpose of comparison. 
In ~\cite{yin2021stability}, quadratic constraints and semidefinite programs are leveraged to obtain  provable inner ROA estimations for feedback systems with NN controllers. This approach amounts to isolating the nonlinearity of the NN policies with local sector constraints which are next used to formulate ROA analysis conditions in the form of linear matrix inequalities (LMIs). The volume of the ellipsoid inner approximations of the ROA can be maximized by solving the resultant LMIs. Although this LMI approach does not scale well for networks over a thousand parameters and cannot be directly applied  to the perception-based setting, it provides a meaningful baseline for the state-feedback setting. We apply our PGD analysis to Example IV.A in ~\cite{yin2021stability} and compare our results with the ROA inner approximations obtained from the LMI approach. From the left plot in Figure \ref{fig:nn_image_phase}, we can see that our PGD-based ROA approximations
are less conservative then the provable inner-estimates obtained from the quadratic constraint approach.
 Both the spherical and ellipsoid ROA approximations found by our approach are larger than the ROA estimate provided by the provable LMI approach. 
This is expected since  our PGD method lets us consider the full nonlinear dynamics of the NN controller and the pendulum without relaxing them via quadratic constraints. 
One downside to our PGD approach is that we must fix the ellipsoid parameterization before-hand and there is no general way to optimize over all possible ellipsoids which scales to arbitrary dimensions. However, even naive approaches like random sampling may provide ellipsoid parameterizations that significantly improve upon the volume of the spherical ROA parametrization.  It is worth emphasizing that our approach provides a much less restrictive and general framework for ROA analysis that does not require careful case-by-case treatment of nonlinearites. Next, we apply our PGD method to the perception-based control setting where the LMI approach cannot be directly applied.


\begin{figure}[ht!]
    \centering
    \includegraphics[width=0.49\textwidth]{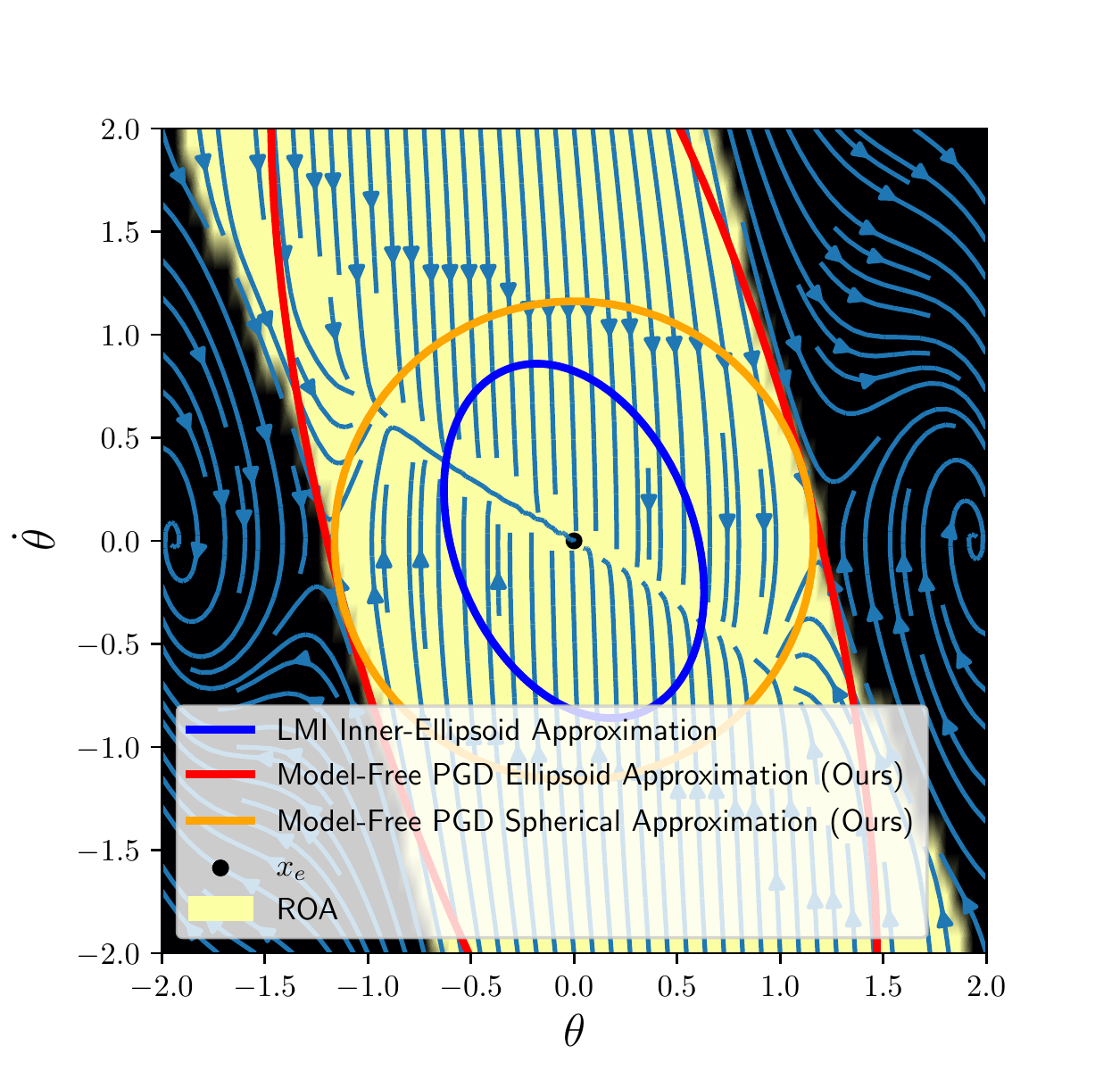}
    \includegraphics[width=0.49\textwidth]{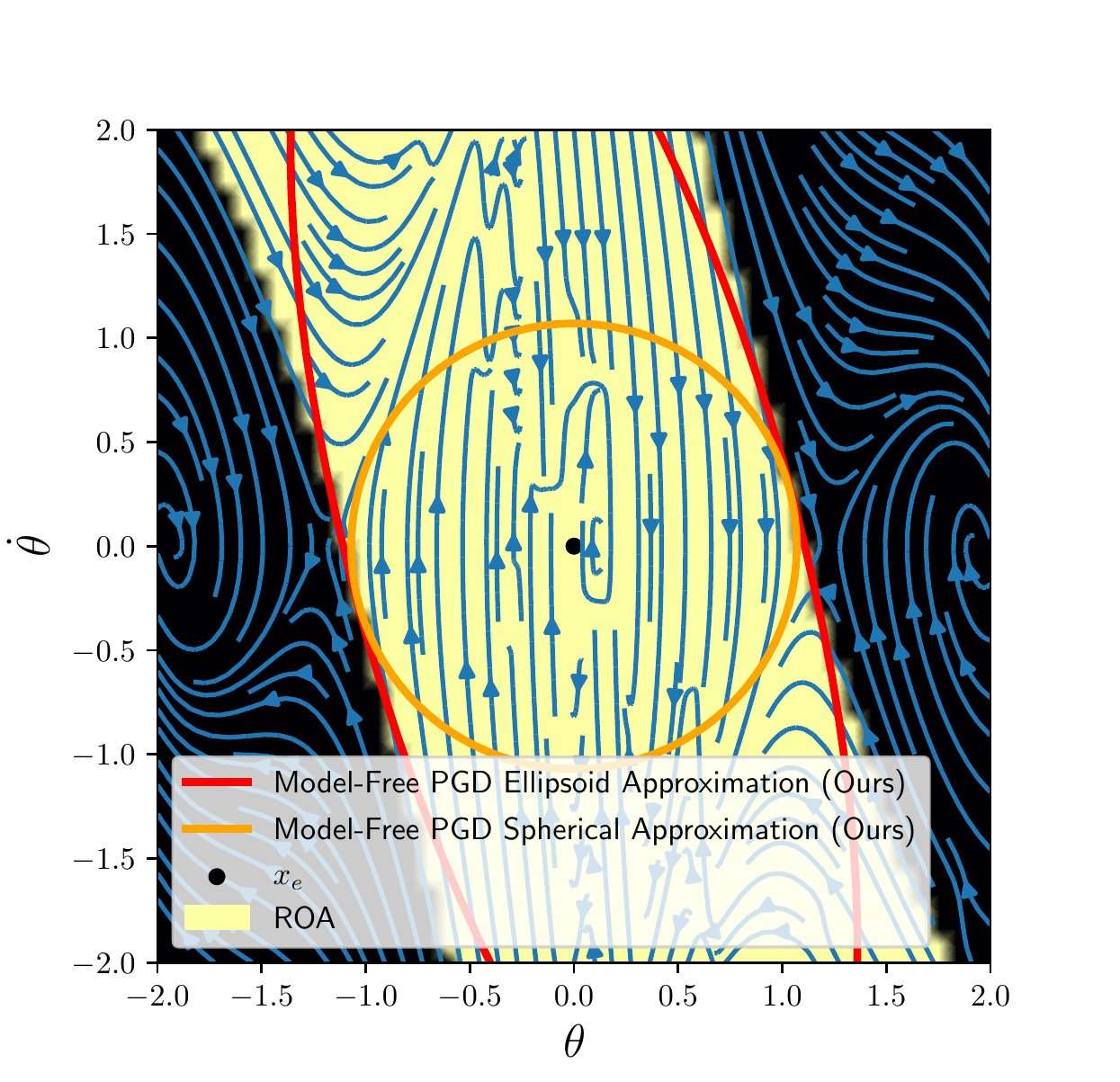}
    \caption{\textbf{Left}: We evaluate the PGD-based approach on the state-feedback inverted-pendulum control for spherical and ellipsoid parameterizations against the true ROA and LMI inner-ellipsoid estimate. The ROA approximation found by our approach subsume the LMI estimate with a volume several times larger. \textbf{Right}: The perception-based NN Controller on inverted-pendulum is evaluated, where the controller maps two previous $84\times 84$-dimensional RGB images to actions. We plot the ROA approximations, which is only compared to the true ROA since the LMI method is not applicable in this situation.}
    \label{fig:nn_image_phase}
\end{figure}

\paragraph{Perception-based model-free ROA analysis.}
Now we apply the PGD approach to compute the ROA estimate in the perception-based control setting where the underlying dynamic model of the inverted-pendulum is still given by~\eqref{eq:inverted_pen_dyn}.
In this case,
using a single RGB image to determine the current control action is not a good idea, since the velocity information is not contained.
We now adopt the NN controller architecture $K$ to be a function of the last two $84\times84$ RGB pixel images of the pendulum.  
It is assumed that the observation $y_t$ is generated by some unknown static function~$h$. We can only simulate the plant which generates these images from the state as depicted in Figure~\ref{fig:image_feedback}. 
For convenience, we use the \textit{Deepmind Control Suite}  to simulate the perception-based inverted-pendulum control system. The model parameters  $(m,l,\mu)$ and the saturation limit are set to be the same as the ones used in the previous state-feedback example.
Again, we use the model-free finite-difference method to estimate the gradient $\nabla L_T$. In this situation, the dimension of the decision variable $\xi=x_0$ is $2$ which is much lower than the dimension of the output variable ($=2 \times 84^2 \times 3 = 42336$). The model-free approach allows us to estimate $\nabla L_T$ without concerning the complexity of the image map $h$. As long as one can simulate the closed-loop dynamics and generate outputs from any initial state, $\nabla L_T$ may be evaluated efficiently in a model-free manner.
Using SAC, we can readily learn a controller that maps the pixel values to actions.  As shown in the right plot of Figure~\ref{fig:nn_image_phase}, the resultant perception-based controller performs similarly to the state-feedback case and has a similar true ROA. 
Based on our PGD analysis, we obtain very similar ROA estimates, despite the complexity of the unknown image generation mapping $h$ and the increased size of the NN architecture (from $32$ hidden units to $256$). Although our ROA estimate is not a provable inner-estimate, this example shows that we are able to handle quite general cases without introducing conservatism.

\subsection{Cart-Pole Systems with Single and Double Links}
\label{sec:cartpole}

By adding a translation degree of freedom to the pendulum example and considering additional rotational links, we can obtain a higher-dimensional cart-pole system with up to $6$ states. This is a useful example since there is an obvious equilibrium point at the upright position which can be studied. Similarly to the pendulum experiment, we study this cart-pole system with state-feedback and perception-based controllers. For convenience, we use the \textit{Deepmind Control Suite} with its default setting of model parameters to simulate this cart-pole example. To capture the physical limit of the actuators, we set the saturation limit of the control action to be $1$.
Again, we adopt the same NN architecture from~\cite{yarats2021mastering} for the purpose of perception-based control.
The controller is trained using SAC, where the objective is to balance the double-link cartpole starting near the equilibrium point. In order to avoid issues with discontinuities in the rotational states, the Euclidean representation of each pole angle is instead given by $(\cos\theta_i,\sin \theta_i)$.
Note that, although it may be possible to apply the LMI approach from~\cite{yin2021stability} to this cartpole system in the state-feedback setting, we would need to formulate system-specific quadratic constraints to handle the nonlinear dynamics. For the NN controllers we consider, the LMI would be quite large, leading to scalability issues. Therefore, we skip the comparison with the LMI approach for this example.
Using a spherical ROA parameterization (i.e. $C=I$), we obtain ROA estimates via the PGD approach ~\eqref{eq:PI} for both the state-feedback and perception-based controllers. 
\paragraph{Single-link case.}
For the single-link cartpole system, we perform a very similar analysis to the inverted-pendulum experiment. For simplicity, we only consider spherical estimates. In order to visualize the ROA, we take a two-dimensional slice through the equilibrium point, fixing the cart position and velocity to be zero. We use $T=200$ and $\delta=10^{-1}$ for the approximate ROA and run the model-free finite-difference method for both state-feedback and image-based experiments. 
The results are given in Figure \ref{fig:cartpole_roa_compare}.
For the the state-feedback controller, we obtain an ROA estimate with radius $\hat r = 0.181$. For the image-based controller, our analysis leads to a slightly improved ROA radius of $\hat r = 0.283$. 
From Figure \ref{fig:cartpole_roa_compare}, we can see that both estimates are reasonable although they are not rigorous inner approximations of the true ROAs.
The true ROAs for the state-feedback and image-based controllers look different, and this is expected due to the different policy paramterization used in these two settings. From Figure \ref{fig:cartpole_roa_compare}, the estimated ROAs from our approach  
 contain points which do not belong to the true ROAs. It seems that this is due to either the gap between the true ROA and the $(T,\delta)$-AROA or the existence of local solutions for \eqref{eq:maxGT}.
 Our ROA estimates are still meaningful in capturing the resilience property of the perception-based controllers.

\begin{figure}
    \centering
    \includegraphics[width=0.49\textwidth]{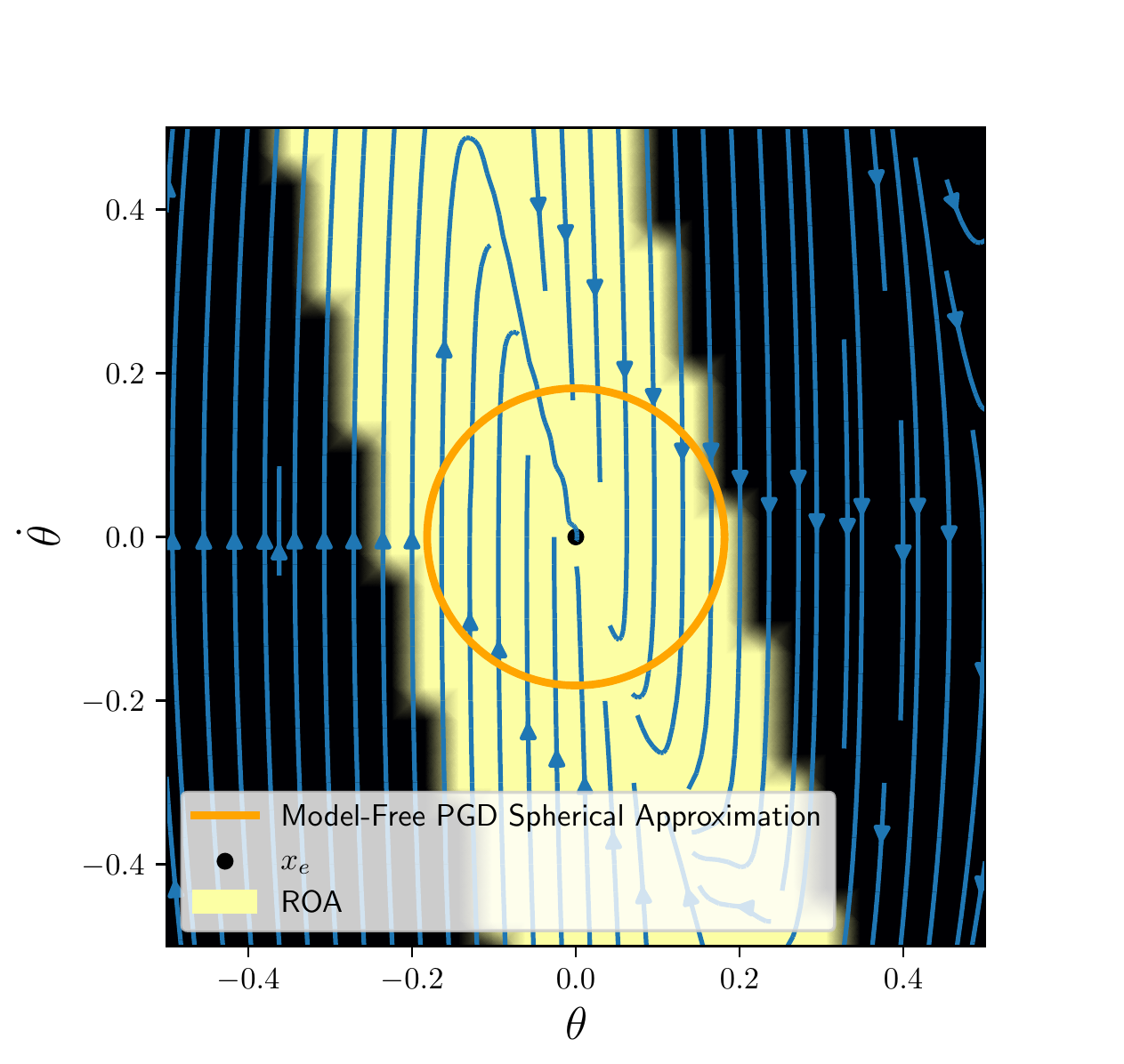}
    \includegraphics[width=0.49\textwidth]{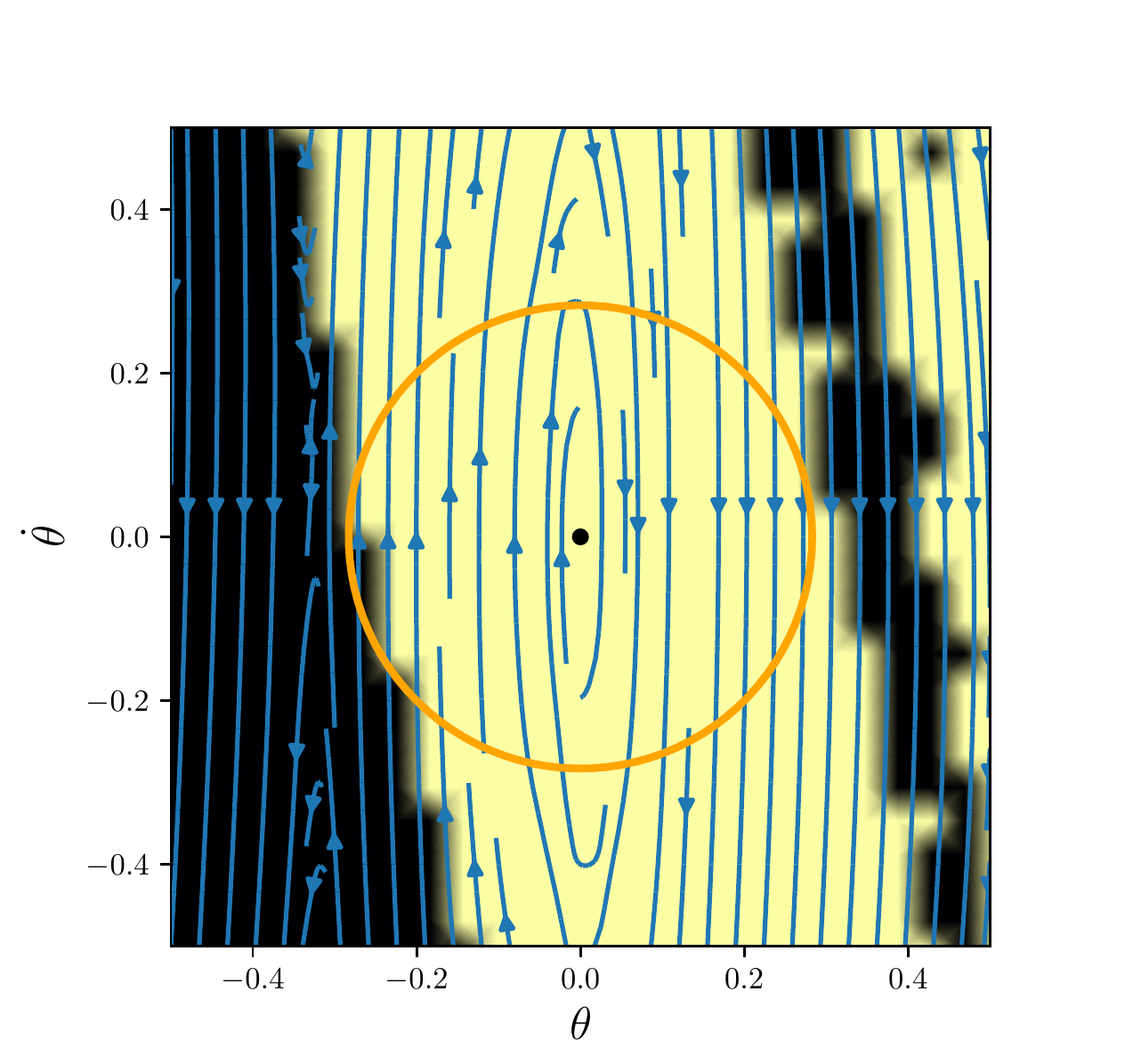}
    \caption{We compare the ROA spherical approximations found by our approach in both state-feedback and perception-based settings on the single cart-pole systems. On the \textbf{left} we have the ROA estimate for state-based controller and on the \textbf{right} is the image-based feedback ROA estimate found by model-free method. This two-dimensional slice of the four-dimensional ROA is determined by fixing the cart position and velocity to be zero.}
    \label{fig:cartpole_roa_compare}
\end{figure}
\paragraph{Double-link case.}
We move one step further and consider the cartpole system with double links. It is expected that the double-link cartpole system is more difficult to stabilize and the resultant ROA should be smaller. 
Again, we use the default setting in the Deepmind Control Suite for the purpose of simulation.
Because this $6$-dimensional ROA estimate is difficult to visualize, here we simply report the radius for the spherical ROA approximations obtained. For the state-feedback controller, we have the estimate $\hat r = 0.034$. The image-based controller yielded a slightly improved ROA radius of $\hat r = 0.077$. Both estimates are of relatively small value when compared to the pendulum example. However, unlike the pendulum, the double-link cart-pole is difficult to stabilize in the first place, since it is under-actuated and highly-sensitive to changes in the initial condition. Based on this intuition, the obtained ROA estimates are reasonable and as expected.

\subsection{Cubic Systems with High-Dimensional States}\label{sec:cubic}

The previous examples demonstrated the capability of our proposed PGD method in addressing unknown image generation mapping $h$ and large-scale NN policies. However, the system states in those examples are low dimensional. 
Here, we revisit the cubic system example from~\cite{cubic_example} to showcase the scalability of our proposed method when the system state is high-dimensional.
 Specifically, we consider
the system $x_{t+1}=\tilde{f}(x_t)$, which is discretized from the continuous-time cubic ODE
$\dot x(t) = (1 - x(t)^\top M x(t))F x(t)$,
where $M$ can be any positive definite (PD) matrix, and $F=-I$. We choose the sampling time as $dt=0.1$. Therefore, $\tilde{f}$ generates $x_{t+1}$ by solving for the state of the cubic continuous-time ODE at $dt$ starting from initial time $0$ and initial state $x_t$. One can simulate $\tilde{f}$ using the the Runge-Kutta method.
This cubic example
has a known ROA given by the ellipsoid $\mathcal{R} = \{x_0 \in \mathbb{R}^{n_x} : x_0^\top M x_0 < 1 \}$, and
was used to study the scalability of SOS \citep{cubic_example}. 
It is found that SOS has difficulty when the state dimension exceeds $10$. As a matter of fact, the original results in \cite{cubic_example} only cover the case when $n_x$ is up to $8$.
As shown in Figure \ref{fig:roa_scaling} (which will be explained in the next paragraph), we will demonstrate that our PGD analysis can be applied when the state dimension is above $1000$.

\begin{figure}[ht!]
    \centering
    \includegraphics[width=1.0\textwidth]{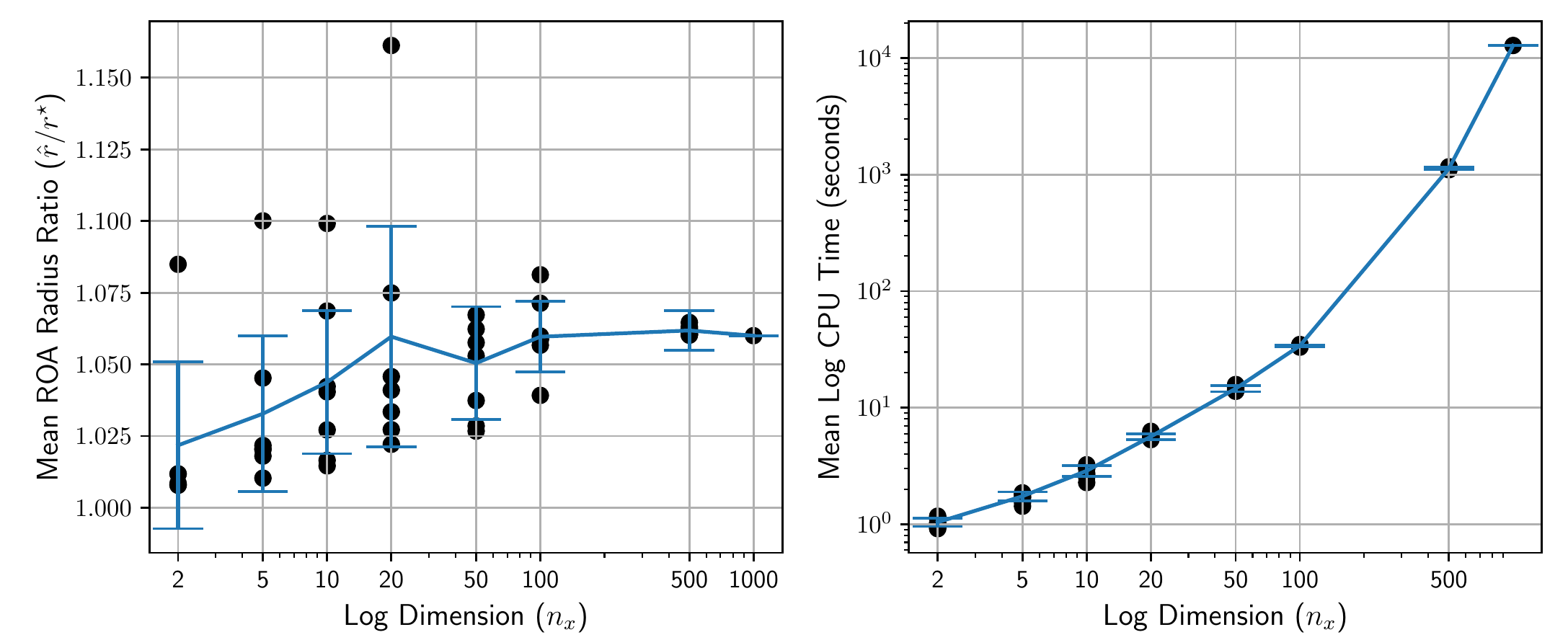}
    \caption{\textbf{left}: For each state dimension $n_x$, we report the mean and one standard deviation over $N=10$ samples for the spherical ROA approximation found by our method. Here $\hat r$ is the estimate and $r^\star$ is the true largest spherical region about the origin that is a subset of the true ellipsoidal ROA. \textbf{right}: The mean CPU time is reported with a $\log$-$\log$ scaling. Note that only a single sample for $n_x = 1000$ was taken and no statistics are reported.}
    \label{fig:roa_scaling}
\end{figure}

Now we present the details for our scalability study.
If we let $M$ be an arbitrary PD matrix and choose $C=I$ in our ROA analysis, then the best inner approximation of the true ROA is given by  a sphere of radius $r^\star := \lambda_{\max}^{-1/2}$, where $\lambda_{\max}$ is the largest eigenvalue of $M$. In this way, we can obtain a spherical ROA ground-truth for systems of arbitrarily high dimension. We can then benchmark our PGD approach by computing the radius $\hat r$ of our spherical ROA approximation which can be compared with $r^\star$. Our ROA approximation will be good if $\frac{\hat r}{r^*}$ is close to $1$. If $\frac{\hat r}{r^*}$ is smaller than $1$, then we have obtained an inner-approximation of the true spherical ROA.
Otherwise our analysis results will be outer approximations of the true ROA. 
This will allow us to demonstrate high-dimensional state cases that can not be solved by SOS but still have a known ground-truth ROA to compare against. 
 In principle, any Hurwitz matrix $F$ and PD matrix $M$ may be used for this example as long as $(F,M)$ satisfy the Lyapunov equation $F^\top M + M F = -Q$ for some PD matrix $Q$. For simplicity and freedom to choose arbitrary ROA parameterization $M$, we use $F=-I$.
 For this experiment, we apply our PGD analysis with finite-difference gradient estimates to randomly generated cases with different state dimensions, i.e. $n_x\in\{2,5,10,20,50,100,500,1000\}$. We fix $T=100$ and $\delta=10^{-2}$.
For each case, we compute the radius $\hat{r}$ and plot the ratio $\frac{\hat{r}}{r^\star}$ in Figure \ref{fig:roa_scaling}, which demonstrates that our approach scales reasonably well as the state dimension increases.
For the state-dimensions up to $500$, we take $10$ randomly generated samples for each dimension. We then take the mean and standard deviations of the approximate radius found by our PGD method, $\hat r$, normalized by the best inner approximation of the true ROA $r^\star$. 
 We also run the PGD method for the case of $n_x=1000$, but only take a single sample for the sake of exposition. This data, along with the mean CPU time taken to run the PGD method on local laptop machine can be found in Figure~\ref{fig:roa_scaling}. It is impressive that the analysis for this case can be run on a laptop within four hours.

 Though our analysis tends to slightly over-estimate the ROA due to our choice of $\delta$, the performance of estimate in terms of mean and variance is more or less consistent. This experiment shows that our method is reasonably effective for high-dimensional problems with failure cases that may be addressed with more careful case-by-case hyper-parameter tuning.
 


\section{Concluding Remarks and Future Work}

In this paper, we tailor the PGD attack as a general-purpose analysis tool for feedback systems with large-scale NN policies and/or high-dimensional sensing modalities such as cameras. We reformulate the ROA analysis as a constrained maximization problem. Such a formulation is not Lyapunov-based, and hence yields a minimum requirement on the problem structure. Then PGD and the model-free variant based on finite difference estimation can be applied to address the ROA analysis in the presence of large-scale NN policies, high-dimensional image observation, and complex coupling between states at different time steps. 

An important future task is to extend the PGD attack for input-output gain analysis which is important for robust control. Recently, the $\mathcal{H}_\infty$ input-output gain has been used in robust reinforcement learning \citep{han2019RL,zhang2020policy,zhang2020stability,donti2020enforcing,zhang2021derivative}. It has been shown that an efficient input-output gain analysis can be combined with adversarial reinforcement learning \citep{pinto2017robust} to improve the robustness in the linear control setting \citep{keivan2021model}. A general-purpose input-output gain analysis will play a crucial role for the developments of robust DRL in the nonlinear or perception-based setting. Hence our next step is to investigate the connections between PGD attacks and input-output gain analysis. One main challenge is that the decision variable dimension of the worst-case disturbance in an input-output gain analysis will be proportional to the approximation window length $T$. This may cause a scalability issue when applying PGD to search for the worst-case disturbance.

\acks{D. Keivan and G. Dullerud are partially funded by NSF
under the grant ECCS 19-32735. A. Havens and B. Hu are
generously supported by the NSF award CAREER-2048168
and the 2020 Amazon research award. P. Seiler is supported
by the US ONR grant N00014-18-1-2209.}

\bibliography{bibtex}

\clearpage
\appendix
\begin{center}
{\Large\bf Supplementary Material}
\end{center}

\section{More Discussion on PGD with Alternative Norm Projections}
\label{sec:A1}
It is well understood that the projection to the $\ell_p$ ball ($p=1,2,\infty$) can be efficiently computed, and hence the update rule \eqref{eq:PGD} can be easily implemented for all three cases. Now we discuss how to implement the update rule \eqref{eq:expan2} for the cases where $p=1$ or $\infty$.

When $p=1$, we need to maximize a linear function subject to an equality constraint $\norm{C\xi}_1=r$. We denote $\tilde{\xi}=C\xi$. Then the update rule \eqref{eq:expan2} is equivalent to
\begin{align*}
\xi^{k+1}=C^{-1}\argmax_{\norm{\tilde{\xi}}_1=r}\left\{\nabla L_T(\xi^k)^\top C^{-1} \tilde{\xi}\right\}
\end{align*}
The above problem yields a simple closed-form solution. Suppose the $i$-th entry of $C^{-\top}\nabla L_T(\xi^k)$ has the largest absolute value among all entries of $C^{-\top}\nabla L_T(\xi^k)$. Then we have $\xi^{k+1}=C^{-1} \tilde{\xi}^{k+1}$ where $\tilde{\xi}^{k+1}$ is a vector whose entries are all $0$ except the $i$-th entry which has a magnitude equal to $r$ and the same sign as the $i$-th entry of $C^{-\top}\nabla L_T(\xi^k)$.

When $p=\infty$, \eqref{eq:expan2} is equivalent to the following problem
\begin{align*}
\xi^{k+1}=C^{-1}\argmax_{\norm{\tilde{\xi}}_\infty=r}\left\{\nabla L_T(\xi^k)^\top C^{-1} \tilde{\xi}\right\}
\end{align*}
which also has an analytical solution. Specifically, we have $\xi^{k+1}=C^{-1} \tilde{\xi}^{k+1}$ where the $j$-th entry of $\tilde{\xi}^{k+1}$ ($j=1,2,\ldots,n_x$) has a magnitude $r$ and the same sign as the $j$-th entry of $C^{-\top}\nabla L_T(\xi^k)$.


\begin{figure}[b!]
    \centering
    \includegraphics[width=1.0\textwidth]{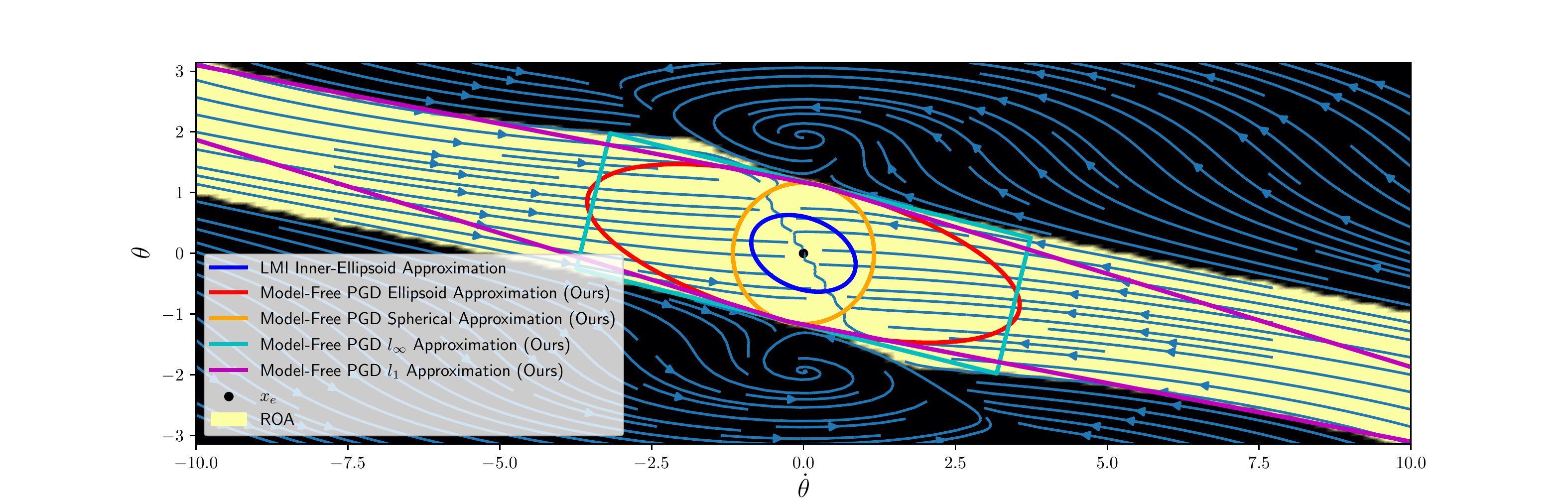}
    \caption{Using the same inverted-pendulum system in feedback with a NN controller, we examine the ROA estimates generated by our PGD-based method using $\ell_1$ and $\ell_\infty$ norm constraints. Since we are only in two dimensions, we can search for scaling parameters that significantly improve on the spherical and ellipsoid estimates, especially in the case of the $\ell_1$ norm constraint.}
    \label{fig:l1_linf}
\end{figure}

It is worth considering $\ell_1$ and $\ell_\infty$ norms, since in many cases, especially in higher dimensions, these norms can yield greater volume ROA estimates. By revisiting the inverted-pendulum case of section~\ref{sec:inverted_pendulum}, we can observe significant improvements for the ROA using the $\ell_1$ and $\ell_\infty$ norm constraints. These results are presented in Figure~\ref{fig:l1_linf}.

\section{Comparison of Convergence for PGD Implementations}
In section~\ref{sec:pgd_attack}, two possible implementations are proposed for finding the worst-case initial condition that maximizes the norm of the final state as in problem~\eqref{eq:maxGT}. The first approach in equation~\eqref{eq:PGD} is the standard PGD method, where the next initial condition is determined by taking a step in the direction of the gradient and then projecting back to the feasible set. The second approach approximates the cost objective using a linear expansion about the previous iterate and provides a closed-form solution~\eqref{eq:PI} when the initial condition is constrained to the boundary of the ellipsoid region. Throughout our results so far, we have only used the later closed-form update rule.

In this section, we provide further empirical study on the performance of each approach and how their convergence behavior relates to the choice of the finite time horizon $T$. Using the inverted-pendulum system~\eqref{eq:inverted_pen_dyn} in feedback with the NN state-feedback policy $K$, we evaluate each approach given the same initial condition guess $x_0$ for an ellipsoid region which just barely intersects the unstable region. 
In figure~\ref{fig:pgd_vs_power}, it can be shown that the convergence speed of the update rule \eqref{eq:PGD} appears to be much more sensitive to the time horizon $T$, since larger time-horizons effectively make the gradient magnitude smaller as the corresponding final states begin to converge to the equilibrium asymptotically. On the other-hand, the update rule \eqref{eq:PI} appears to not depend on $T$ and converges faster over all choices of time-horizon $T$. This is beneficial to our application since $T$ determines, in some sense, the accuracy of our notion of asymptotic stability and in turn the accuracy of the ROA estimate.

\begin{figure}[b!]
    \centering
    \includegraphics[width=0.95\textwidth]{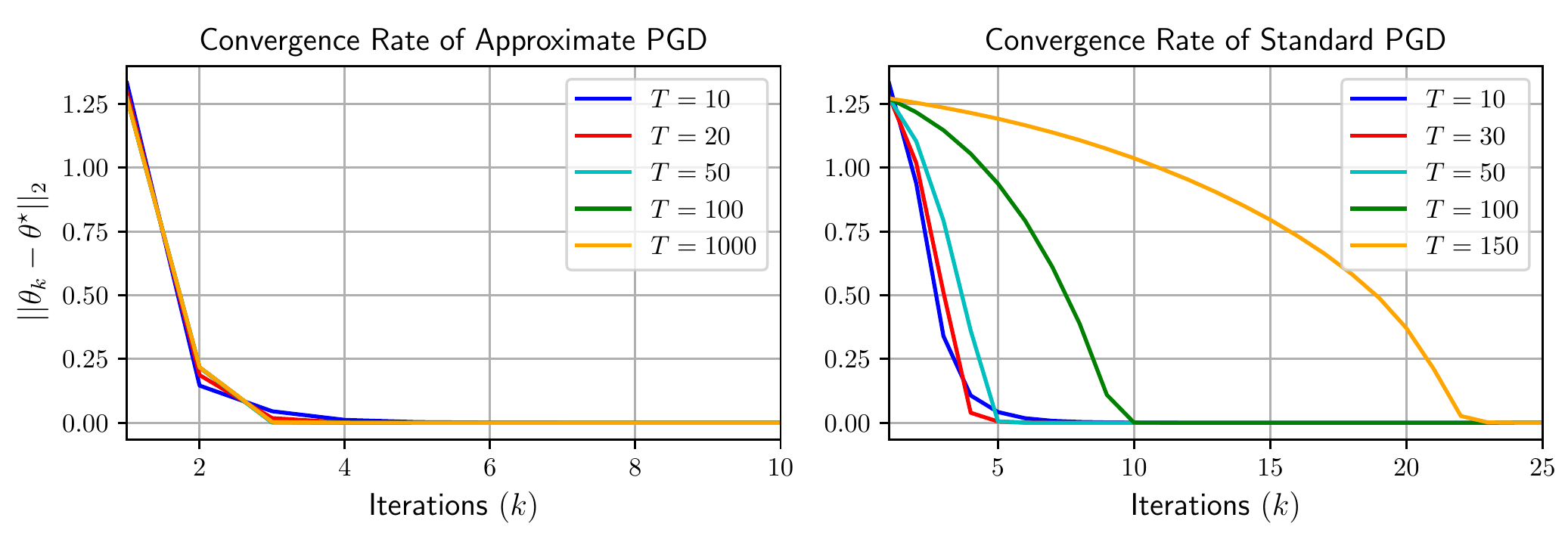}
    \caption{The convergence behavior of the standard PGD update rule~\eqref{eq:PGD} and the approximate closed-form PGD update~\eqref{eq:PI} are compared. For both methods, we use the same inverted-pendulum with state-feedback NN controller and an identical initial condition. The $y$-axis denotes the euclidean distance from the current worst-case initial state guess to the nearest unstable initial state. \textbf{left}: The closed-form update rule is evaluated for several time horizons $T$ and convergence roughly at the same rate independent of $T$. The closed-form update is generally faster than the PGD approach for all choices of $T$. \textbf{right}: Standard PGD is evaluated for several $T$, however the number of iterates required to reach a solution quickly grows with $T$. This is primarily because, as $T$ grows, the gradient at a given initial state in the ROA will become smaller as the trajectory converges to the equilibrium point.}
    \label{fig:pgd_vs_power}
\end{figure}

\end{document}